\documentclass[11pt]{article}
\usepackage{mathbbold}
\usepackage{amsfonts}
\usepackage{mathrsfs}
\usepackage{bbm}
\usepackage{latexsym,amsfonts,amssymb,amsmath,amsthm}
\usepackage{graphicx}

\usepackage{setspace}

\newtheorem{theorem}{Theorem}[section]
\newtheorem{definition}[theorem]{Definition}
\newtheorem{lemma}[theorem]{Lemma}

\newtheorem{remark}[theorem]{Remark}
\newtheorem{prop}[theorem]{Proposition}
\numberwithin{equation}{section}
\begin{document}

\baselineskip 20pt

\begin{center}

\textbf{\Large Random Attractor For Stochastic Lattice
FitzHugh-Nagumo System Driven By $\alpha$-stable L\'evy Noises}
\footnote{This work has been partially supported by NSFC Grants
11071199, GXNSF Grants
2013GXNSFBA019008 and GXPDRP Grants 2013YB102.\\
$^*$Corresponding author: A. Gu~(mathgah@gmail.com).}

\vskip 0.5cm

{\large  Anhui Gu, Yangrong Li and Jia Li}

\vskip 0.3cm

\textit{School of Mathematics and Statistics, Southwest
University, Chongqing, 400715, China}\\

\vskip 1cm

\begin{minipage}[c]{15cm}

\noindent \textbf{Abstract}: The present paper is devoted to the
existence of a random attractor for stochastic lattice
FitzHugh-Nagumo system driven by $\alpha$-stable L\'evy noises under
some dissipative conditions. \vspace{5pt}

\vspace{5pt}


\end{minipage}
\end{center}

\vspace{10pt}
%\newpage

\baselineskip 18pt

\section{Introduction}
\noindent We consider the following stochastic lattice
FitzHugh-Nagumo system (SLFNS)
\begin{equation}\label{sys1}
\left\{
\begin{array}{l}
\frac{du_i}{dt_+}=u_{i-1}-2u_i+u_{i+1}-\lambda u_i+f_i(u_i)-v_i\\
\quad\quad+h_i+\sum_{j=1}^N\varepsilon_j u_i\diamond \frac{dL_t^{j}}{dt},\\
\frac{dv_i}{dt_+}=\varrho u_i-\varpi v_i
+g_i+\sum_{j=1}^N\varepsilon_j v_i\diamond \frac{dL_t^{j}}{dt},\\
 u(0)=u_{0}=(u_{i0})_{i\in \mathbb{Z}},
 v(0)=v_{0}=(v_{i0})_{i\in \mathbb{Z}}
\end{array}
\right.
\end{equation}
where $\mathbb{Z}$ denotes the integer set, $u_i\in \mathbb{R} $,
$\lambda, \varrho$ and $\varpi$ are positive constants, $h_i, \
g_i\in \mathbb{R}$, $f_i$ are smooth functions satisfying some
dissipative conditions, $\varepsilon_j\in \mathbb{R}$ for $j = 1,
...,N$, $L_t^j$ are mutually independent $\alpha$-stable L\'evy
motions ($1<\alpha<2$), and $\diamond$ denotes the Marcus sense in
the stochastic term, , $\frac{d\cdot}{dt_+}$ is right-hand
derivative of $\cdot(t)$ at $t$, $\ell^2=(\ell^2, ( \cdot , \cdot ),
\|\cdot\|)$ denotes the regular space of infinite sequences.

As we all known, noises involved in realistic systems will play an
important role as intrinsic phenomena rather than just compensation
of defects in deterministic models. Stochastic lattice dynamical
systems (SLDS) arise naturally in a wide variety of applications
where the spatial structure has a discrete character and random
influences or uncertainties are taken into account. For the recent
research of SLDS, we can see e.g. \cite{BLL, Huang, CL, ZZ1, HSZ}
for the first- or second-order lattice dynamical systems with white
noises in regular (or weight) space of infinite sequences, see e.g.
\cite{Gu1, Gu2} for the first-order lattice dynamical systems driven
by fractional Brownian motions, see \cite{Gu3} for the first-order
lattice dynamical systems with non-Gaussian noises.

When there are no noises terms, form similar to \eqref{sys1} is the
discrete of the FitzHugh-Nagumo system which arose as modeling the
signal transmission across axons in neurobiology (see \cite{Jones}).
Lattice FitzHugh-Nagumo system was used to stimulate the propagation
of action potentials in myelinated nerve axons (see \cite{EV}).
Gaussian processes like Brownian motion have been widely used to
model fluctuations in engineering and science. When lattice
FitzHugh-Nagumo system perturbed by additive or multiplicative white
noises, the existence of random attractors has been proved in
\cite{Huang, Gu4}. To the best of our knowledge, there are no
results on the system when it is perturbed by a non-Gaussian noise
(in terms of L\'evy noise).

In fact, some complex phenomena involve non-Gaussian fluctuations
with peculiar properties such as anomalous diffusion (mean square
displacement is a nonlinear power law of time) \cite{BG} and heavy
tail distribution (non-exponential relaxation) \cite{Yonezawa}. For
this topic, we can refer to \cite{SZF, SSB, Herrchen, Ditlevsen} for
more details. A L\'evy motion $L_t$ is a non-Gaussian process with
independent and stationary increments, i.e,. increments $\Delta
L_t=L_{t+\Delta t}-L_{t}$ are stationary and independent for any non
overlapping time lags $\Delta t$. Moreover, its sample paths are
only continuous in probability, namely,
$\mathbb{P}(|L_t-L_{t_0}|\geq \epsilon)\rightarrow 0$ as
$t\rightarrow t_0$ for any positive $\epsilon$. With a suitable
modification, these path may be taken as c\`adl\`ag, i.e., paths are
continuous on the right and have limits on the left. This continuity
is weaker than the usual continuity in time. Indeed, a c\`adl\`ag
function has at most countably many discontinuities on any time
interval, which generalizes the Brownian motion to some extent (see
e.g. \cite{Applebaum}). As a special case of L\'evy processes, the
symmetric $\alpha$-stable L\'evy motion plays an important role
among stable processes just like Brownian motion among Gaussian
processes. A stochastic process $\{L_t, t\ge 0\}$ is called the
$\alpha$-stable L\'evy motions if  (i) $L_0=0$ a.e., (ii) $L$ has
independent increments, and (iii) $L_t-L_s\sim
\mathbf{S}_{\alpha}((t-s)^{\frac{1}{\alpha}}, \beta, 0)$ for $0\le
s<t<\infty$ and for some $0<\alpha\le 2, -1\le \beta\le 1$, where
$\mathbf{S}_{\alpha}(\sigma, \beta, \nu)$ denotes the
$\alpha$-stable distribution with index of stability $\alpha$, scale
parameter $\sigma$, skewness parameter $\beta$ and shift parameter
$\nu$; in particular, $\mathbf{S}_{2}(\sigma, 0, \mu)=N(\mu,
2\sigma^2)$ denotes the Gaussian distribution. For more details on
$\alpha$-stable distributions, we can refer to \cite{Sato}. It is
worth mentioning that when $\alpha=2$, we have the standard Brownian
motion, which the Marcus sense stochastic terms (see e.g.
\cite{Marcus}) reduce to the Stratonovich stochastic terms and the
existence of a random attractor for system \eqref{sys1} has been
considered in \cite{Gu4}. For the further development on L\'evy
motions, we can refer to the recent monographs \cite{Applebaum, PZ}.

The goal of this article is to establish the existence of a random
attractor for SLFNS with the nonlinearity $f_i$ under some
dissipative conditions and driven by $\alpha$-stable L\'evy noises
with $\alpha\in (1, 2)$. By virtue of an Ornstein-Uhlenbeck process
with a stationary solution, we transform system \eqref{sys1} into a
conjugated random integral equation (with a solution in the sense of
Carath\'eodory). Here, we assume that $1<\alpha<2$ since this is the
only case where the solutions of the Ornstein-Uhlenbeck equations
for $\alpha$-stable L\'evy noises are stationary, which is vital to
our purpose. Fot the case of $0<\alpha<1$, there will be a new
challenges for us for future research.

The paper is organized as follows. In Sec. 2, we recall some basic
concepts in random dynamical systems. In Sec. 3, we give a unique
solution to system \eqref{sys1} and make sure that the solution
generates a random dynamical system. We establish the main result,
that is, the existence of a random attractor generated by system
\eqref{sys1} in Sec. 4.

\section{Random dynamical systems and random attractors}

For the reader's convenience, we introduce some basic concepts
related to random dynamical systems and random attractors, which are
taken from \cite{Arnold, Chueshov, HSZ}. Let $(\mathbb{E},
\|\cdot\|_{\mathbb{E}})$ be a separable Hilbert space and $(\Omega,
\mathcal{F}, \mathbb{P})$ be a probability space.

\begin{definition}
A stochastic process $\{\varphi(t, \omega)\}_{t\geq 0, \omega\in
\Omega}$ is a continuous random dynamical system (RDS) over
$(\Omega, \mathcal{F}, \mathbb{P},(\theta_{t})_{t\in \mathbb{R}})$
if $\varphi$ is $(\mathcal{B}[0,\infty)\times \mathcal{F}\times
\mathcal{B}(\mathbb{E}), \mathcal{B}(\mathbb{E}))$-measurable, and
for all $\omega \in \Omega$,

(i) the mapping $\varphi(t,\omega): \mathbb{E}\mapsto \mathbb{E}$,
$x\mapsto \varphi(t,\omega)x$ is continuous for every $t\geq 0$,

(ii)  $\varphi(0,\omega)$ is the identity on $\mathbb{E}$,

(iii) \mbox{(cocycle property)} \
$\varphi(s+t,\omega)=\varphi(t,\theta_{s}\omega)\varphi(s,\omega)$
for all $s, t\geq 0$.
\end{definition}

\begin{definition}  \label{tempered random set}
(i) A set-valued mapping $\omega\mapsto B(\omega)\subset \mathbb{E}$
(we may write it as $B(\omega)$ for short) is said to be a random
set if the mapping $\omega\mapsto$ dist$_{\mathbb{E}}(x, B(\omega))$
is measurable for any $x\in \mathbb{E}$, where dist$_{\mathbb{E}}(x,
D)$ is the distance in $\mathbb{E}$ between the element $x$ and the
set $D\subset \mathbb{E}$.

(ii) A random set $B(\omega)$ is said to be bounded if there exist
$x_0\in \mathbb{E}$ and a random variable $r(\omega)>0$ such that
$B(\omega)\subset\{x\in \mathbb{E}: \|x-x_0\|_{\mathbb{E}}\leq
r(\omega), x_0\in \mathbb{E}\}$ for all $\omega \in \Omega$.

(iii) A random set $B(\omega)$ is called a compact random set if
$B(\omega)$ is compact for all $\omega \in \Omega$.

(iv) A random bounded set $B(\omega) \subset \mathbb{E}$ is called
tempered with respect to $(\theta_{t})_{t\in \mathbb{R}}$ if for
a.e. $\omega \in \Omega$, \ $\lim_{t\rightarrow +\infty}e^{-\gamma
t}d(B(\theta_{-t}\omega))=0 \ \ \mbox{for all} \ \ \gamma > 0$,
where $d(B)=\sup_{x\in B}\|x\|_{\mathbb{E}}$. A random variable
$\omega \mapsto r(\omega)\in \mathbb{R}$ is said to be tempered with
respect to $(\theta_{t})_{t\in \mathbb{R}}$ if for a.e. $\omega \in
\Omega$, $\lim_{t\rightarrow +\infty} \sup_{t\in
\mathbb{R}}e^{-\gamma t}r(\theta_{-t}\omega)=0 \ \ \mbox{for all} \
\ \gamma > 0$.

\end{definition}

We consider an RDS $\{\varphi(t, \omega)\}_{t\geq 0, \omega\in
\Omega}$ over  $(\Omega, \mathcal{F}, \mathbb{P},(\theta_{t})_{t\in
\mathbb{R}})$ and $\mathcal{D}(\mathbb{E})$ the set of all tempered
random sets of $\mathbb{E}$.

\begin{definition}
A random set $\mathcal{K}$ is called an absorbing set in
$\mathcal{D}(\mathbb{E})$ if for all $B\in \mathcal{D}(\mathbb{E})$
and a.e. $\omega \in \Omega$ there exists $t_{B}(\omega)>0$ such
that
$$\varphi(t,\theta_{-t}\omega)B(\theta_{-t}\omega)\subset
\mathcal{K}(\omega) \ \ \mbox{for all} \ \ t\geq t_{B}(\omega).$$

\end{definition}

\begin{definition}
A random set $\mathcal{A}$ is called a global random
$\mathcal{D}(\mathbb{E})$ attractor (pullback
$\mathcal{D}(\mathbb{E})$ attractor) for $\{\varphi(t,
\omega)\}_{t\geq 0, \omega\in \Omega}$ if the following hold:

(i) $\mathcal{A}$ is a random compact set, i.e. $\omega\mapsto d(x,
\mathcal{A}(\omega))$ is measurable for every $x\in \mathbb{E}$ and
$\mathcal{A}(\omega)$ is compact for a.e. $\omega \in \Omega$;

(ii)  $\mathcal{A}$ is strictly invariant, i.e. for $\omega \in
\Omega$ and all $t\geq 0$,
$\varphi(t,\omega)\mathcal{A}(\omega)=\mathcal{A}(\theta_{t}\omega)$;

(iii)  $\mathcal{A}$ attracts all sets in $\mathcal{D}(\mathbb{E})$,
i.e. for all $B\in \mathcal{D}(\mathbb{E})$ and a.e. $\omega \in
\Omega$, we have
$$\lim_{t\rightarrow+\infty}d(\varphi(t,\theta_{-t}\omega)
B(\theta_{-t}\omega), \mathcal{A}(\omega))=0,$$ where
$d(X,Y)=\sup_{x\in X} \inf_{y\in Y}\|x-y\|_{\mathbb{E}}$ is the
Hausdorff semi-metric ($X\subseteq \mathbb{E}, Y\subseteq
\mathbb{E}$).\end{definition}

\begin{prop}(See \cite{HSZ} .) \label{condition} Suppose that

(a) there exists a random bounded absorbing set $K(\omega)\in
\mathcal{D}(\ell^2)$, $\omega\in \Omega$, such that for any
$B(\omega)\in \mathcal{D}(\ell^2)$ and all $\omega\in \Omega$, there
exists $T(\omega, B)>0$ yielding $\varphi(t, \theta_{-t}\omega,
B(\theta_{-t}\omega))\subset K(\omega)$ for all $t\geq T(\omega,
B)$;

(b) the RDS $\{\varphi(t, \omega)\}_{t\geq 0, \omega\in \Omega}$ is
random asymptotically null on $K(\omega)$, i.e., for any
$\epsilon>0$, there exist $T(\epsilon, \omega, K)>0$ and
$I_0(\epsilon, \omega, K)\in \mathbb{N}$ such that
\begin{equation}\label{asy null}
\begin{split}
\sup_{u\in K(\omega)}&\sum_{|i|>I_0(\epsilon, \omega, K(\omega))}|
\varphi_i(t, \theta_{-t}\omega, u(\theta_{-t}\omega))|^2\\ &\leq
\epsilon^2,   \quad\quad\quad\forall t\geq T(\epsilon, \omega,
K(\omega)).
\end{split}
\end{equation}

Then the RDS $\{\varphi(t, \omega, \cdot)\}_{t\geq 0, \omega\in
\Omega}$ possesses a unique global random $\mathcal{D}(\ell^2)$
attractor given by
\begin{eqnarray}\label{ran attr}
\mathcal{\tilde{A}}(\omega)=\bigcap_{\tau\geq T(\omega, K)}
\overline{\bigcup_{t\geq \tau} \varphi(t,\theta_{-t}\omega,
K(\theta_{-t}\omega))}.\end{eqnarray}

\end{prop}

\section{SLFNS driven by $\alpha$-stable L\'evy noises}

Let $(\Omega, \mathcal{F}, \mathbb{P})$ be a probability space,
where $\Omega=\mathcal{S}(\mathbb{R}, \ell^2)$ with Skorokhod metric
as the canonical sample space of c\`adl\`ag functions defined on
$\mathbb{R}$ and taking values in $\ell^2$,
$\mathcal{F}:=\mathcal{B}(\mathcal{S}(\mathbb{R}, \ell^2))$ the
associated Borel $\sigma$-field and $\mathbb{P}$ is the
corresponding (L\'evy) probability measure on $\mathcal{F}$ which is
given by the distribution of a two-sided L\'evy process with paths
in $\mathcal{S}(\mathbb{R}, \ell^2)$, i.e. $\omega(t)=L_t(\omega)$.
Let $\theta_{t}\omega(\cdot)=\omega(\cdot+t)-\omega(t),\  t\in
\mathbb{R},$ then the mapping $(t, \omega)\rightarrow
\theta_t\omega$ is continuous and measurable (see \cite{Arnold}),
and the (L\'evy) probability measure is $\theta$-invariant, i.e.
$\mathbb{P}(\theta_t^{-1}(\tilde{A}))=\mathbb{P}(\tilde{A})$ for all
$\tilde{A}\in \mathcal{F}$ (see \cite{Applebaum}).

For convenience, we now formulate system \eqref{sys1} as a
stochastic differential equation in $\ell^2\times\ell^2$. For
$u=(u_i)_{i\in \mathbb{Z}}\in \ell^2$, define $\mathbb{A},
\mathbb{B}, \mathbb{B}^*$ to be linear operators from $\ell^2$ to
$\ell^2$ as follows:
\begin{eqnarray*}
(\mathbb{A}u)_i&=&-u_{i-1}+2u_i-u_{i+1}, \nonumber\\
(\mathbb{B}u)_i&=&u_{i+1}-u_i, \ \ (\mathbb{B}^*u)_i=u_{i-1}-u_i,\ \
i\in \mathbb{Z}.
\end{eqnarray*}
It is easy to show that
$\mathbb{A}=\mathbb{B}\mathbb{B}^*=\mathbb{B}^*\mathbb{B}$, $(
\mathbb{B}^* u, u')=(u, \mathbb{B} u')$ for all $u, u'\in \ell^2$,
which implies that $(\mathbb{A}u, u)\geq 0$.

Let $f_i\in \mathcal{C}(\mathbb{R})$ satisfy the conditions that
$\sup_{i\in \mathbb{Z}}|f'(u)|$ is bounded for $u$ in bounded sets
and $f_i(x)x\geq 0$ for all $x\in \mathbb{R}$. Let $\tilde{f}$ be
the Nemytski operator associated with $f_i$, for $u=(u_i)_{i\in
\mathbb{Z}}\in \ell^2$, then $\tilde{f}(u)\in \ell^2$ and
$\tilde{f}$ is locally Lipschitz from $\ell^2$ to $\ell^2$ (see
\cite{BLL, CL}). In the sequel, when no confusion arises, we
identify $\tilde{f}$ with $f$.

Let $\mathbb{E}=\ell^2\times\ell^2$, for $\Psi=(u, v)\in
\mathbb{E}$, denote the norm
$\|\Psi\|^2:=\|\Psi\|^2_{\mathbb{E}}=\|u\|^2+\|v\|^2$. Then system
\eqref{sys1} can be interpreted as a system of integral equations in
$\mathbb{E}$ for $t\in \mathbb{R}$ and $\omega\in \Omega$,
\begin{equation}\label{sys2}
\left\{
\begin{array}{l}
u(t)=u(0)+\int_0^t(-\mathbb{A}u(s)-\lambda u(s)\\
\quad\quad+f(u(s))-v(s)
+h)ds\\
\quad\quad\quad+\sum_{j=1}^N\int_0^t\varepsilon_j u(s) \diamond
dL_t^{j},\\
v(t)=v(0)+\int_0^t(\varrho u(s)-\varpi v(s)+g)ds\\
\quad\quad+\sum_{j=1}^N\int_0^t\varepsilon_j v(s) \diamond dL_t^{j},
\end{array}
\right.
\end{equation}
where the stochastic integral is understood to be in the Marcus
sense.

To prove that this stochastic equation \eqref{sys2} generates a
random dynamical system, we will transform it into a random
differential equation in $\mathbb{E}$. Now, we introduce the
Ornstein-Uhlenbeck processes in $\ell^2$ on the metric dynamical
system $(\Omega, \mathcal{F}, \mathbb{P},(\theta_{t})_{t\in
\mathbb{R}})$ given by the random variable
\begin{equation}\label{OU-sol}
z(\theta_{t}\omega)=-\int^0_{-\infty} e^{s} \theta_{t}\omega(s) ds,
\ \ t\in \mathbb{R}, \omega\in \Omega.
\end{equation} The above integrals exist in the sense
of any path with a subexponential growth, and $z$ solves the
following Ornstein-Uhlenbeck equation
\begin{equation}\label{OU-equ}
dz+zdt=dL_t,  \ \ t\in \mathbb{R}.
\end{equation}
In fact, we have the following properties (see Lemma 3.1 in
\cite{Gu3}): (i) There exists a $\{\theta_{t}\}_{t\in
\mathbb{R}}$-invariant subset $\bar{\Omega}\in \mathcal{F}$ of full
measure for a.e. $\omega\in \bar{\Omega}$, the random variable
\begin{equation*}
z(\omega)=-\int^0_{-\infty} e^{s} \omega(s) ds,
\end{equation*}
is well defined and the unique stationary solutions of
\eqref{OU-equ} is given by \eqref{OU-sol}. Moreover, the mapping
$t\rightarrow z(\theta_t\omega)$ is c\`adl\`ag; (ii) For $ \omega
\in \bar{\Omega}$, the sample paths $\omega(t)$ of $L_t$ satisfy
\begin{equation*}
\lim_{t\rightarrow\pm \infty}\frac{\omega(t)}{t}=0, \ t\in
\mathbb{R}
\end{equation*}
and
\begin{equation*}
\lim_{t\rightarrow\pm\infty}\frac{|z(\theta_t\omega)|}{|t|}=
\lim_{t\rightarrow\pm\infty}\frac{1}{t}\int_0^t
z(\theta_t\omega(s))ds=0.
\end{equation*}

Now, let $z_j$ be the associated Ornstein-Uhlenbeck process
corresponding to \eqref{OU-equ} with $L_t^{j}$ instead of $L_t$ and
denote $\Lambda(\omega)=e^{\sum_{j=1}^N\varepsilon_j
z_j(\omega)}\mathbf{Id}_{\mathbb{E}}$, then $\Lambda(\omega)$ is
clearly a homeomorphism in $\mathbb{E}$ and the inverse operator is
well defined by $\Lambda^{-1}(\omega)=e^{-\sum_{j=1}^N\varepsilon_j
z_j(\omega)}\mathbf{Id}_{\mathbb{E}}$. It is easy to verify that
$\|\Lambda^{-1}(\theta_t\omega)\|$ has sub-exponential growth as
$t\rightarrow\pm \infty$ for $\omega\in \Omega$. Hence
$\|\Lambda^{-1}\|$ is tempered. Since the mapping of $\theta$ on
$\bar{\Omega}$ has the same properties as the original one if we
choose the trace $\sigma$-algebra with respect to $\bar{\Omega}$ to
be denoted also by $\mathcal{F}$, we can change our metric dynamical
system with respect to $\bar{\Omega}$, and still denoted the symbols
by $(\Omega, \mathcal{F}, \mathbb{P},(\theta_{t})_{t\in
\mathbb{R}})$.

Denote $\xi(\theta_t\omega)=\sum_{j=1}^N\varepsilon_j
z_j(\theta_t\omega)$, and consider the change in variables
\begin{equation*}
\begin{split}
(U(t), V(t))=&\Lambda^{-1}(\theta_t\omega)(u(t), v(t))\\
\quad\quad&=e^{-\xi(\theta_t\omega)}(u(t), v(t)),
\end{split}
\end{equation*}
where $(u, v)$ is the solution of \eqref{sys2}, then we get the
evolution equations with random coefficients but without white noise
\begin{equation}\label{sys3}
\left\{
\begin{array}{l}
\frac{dU}{dt_+}=-\mathbb{A}U-(\lambda-\xi(\theta_t\omega))U\\
\quad\quad+e^{-\xi(\theta_t\omega)}f(e^{\xi(\theta_t\omega)}U)
-V+e^{-\xi(\theta_t\omega)}h,\\
\frac{dV}{dt_+}=\varrho U-(\varpi-\xi(\theta_t\omega))V
+e^{-\xi(\theta_t\omega)}g\\
\end{array}
\right.
\end{equation}
and initial condition $(U(0), V(0))=(U_0, V_0)\in \mathbb{E}$.

Now, we have the following result:
\begin{theorem}\label{uniqueness}
Let $T>0$ and $\Psi_0=(U_0, V_0)\in \mathbb{E}$ be fixed, then the
following statements hold:

(i) For every $\omega\in \Omega$, system \eqref{sys3} has a unique
solution $\Psi(\cdot, \omega, \Psi_0)=(U(\cdot, \omega, U_0),
V(\cdot, \omega, V_0))\in \mathcal{C}([0, T), \mathbb{E})$ in the
sense of Carath\'eodory.

(ii) For each $\omega\in \Omega$, the mapping $\Psi_0\in
\mathbb{E}\mapsto \Psi(\cdot, \omega, \Psi_0)\in \mathcal{C}([0, T),
\mathbb{E})$ is continuous, which implies the solution $\Psi$ of
\eqref{sys3} continuously depends on the initial data $\Psi_0$.

(iii) Equation \eqref{sys3} generates a continuous RDS
$(\varphi(t))_{t\geq 0}$ over $(\Omega, \mathcal{F},
\mathbb{P},(\theta_{t})_{t\in \mathbb{R}})$, where $\varphi(t,
\omega, \Psi_0)=\Psi(t, \omega, \Psi_0)$ for $\Psi_0\in \mathbb{E}$,
$t\geq 0$ and for all $\omega\in \Omega$. Moreover, $\psi(t, \omega,
\Psi_0)=\Lambda(\theta_t\omega)\psi(t, \omega,
\Lambda^{-1}(\omega)\Psi_0)$ for $\Psi_0\in \mathbb{E}$, $t\geq 0$
and for all $\omega\in \Omega$, then $\psi$ is another RDS for which
the process $(\omega, t) \rightarrow (\psi(t, \omega, \Psi_0))$
solves \eqref{sys2} for any initial condition $\Psi_0\in
\mathbb{E}$.
\end{theorem}

\begin{proof}
(i) Let $F(t,
U)=e^{-\xi(\theta_t\omega)}f(e^{\xi(\theta_t\omega)}U)$, for any
fixed $T>0$ and $\Psi_0\in \mathbb{E}$ and let $U_1, U_2\in Y$,
where $Y$ is a bounded set in $\mathbb{E}$, we have
\begin{equation*}
\begin{split}
\|F(t, &U_1)-F(t, U_2)\|\\
&=\|e^{-\xi(\theta_t\omega)} f(e^{\xi(\theta_t\omega)}U_1)-
e^{-\xi(\theta_t\omega)}f(e^{\xi(\theta_t\omega)}U_2)\|\\
&\quad\quad\quad\quad\leq C_Y\|v_1-v_2\|,
\end{split}
\end{equation*}
where $C_Y$ is a constant only depending on $Y$. This implies that
the mapping $F(t, U)$ is locally Lipschitz with respect to $U$ and
the Lipschitz constant is uniformly bounded in $[0, T]$. By the
standard arguments, we know that \eqref{sys3} possesses a local
solution $\Psi(\cdot, \omega, \Psi_0)\in \mathcal{C}([0, T_{\max}),
\mathbb{E})$, where $[0, T_{\max})$ is the maximal interval of
existence of the solution of \eqref{sys3}. Next, we need to show
that the local solution is a global one. By taking the inner
products of $U$ and $V$ respectively in $\ell^2$ with the two
equations in system \eqref{sys3}, we have
\begin{equation}\label{est1}
\begin{split}
\frac{d}{dt_+}(&\|U\|^2+\frac{1}{\varrho}\|V\|^2)\\
&\le
-(\delta-2\xi(\theta_t\omega))(\|U\|^2+\frac{1}{\varrho}\|V\|^2)\\
&\quad\quad+\frac{1}{\delta}(\|h\|^2+\frac{1}{\varrho}\|g\|^2)
e^{-2\xi(\theta_t\omega)},
\end{split}
\end{equation}
where $\delta=\min\{\lambda, \varpi\}$. By virtue of the special
Gronwall lemma (see Lemma 2.8 in \cite{Robinson}), it yields that
\begin{equation*}
\begin{split}
\|\Psi(t)&\|^2\leq e^{-\delta t+2\int_0^t\xi(\theta_s\omega)ds}
\|\Psi_0\|^2+c_1(\|h\|^2+\|g\|^2)\\
&\cdot e^{-\delta t+2\int_0^t\xi(\theta_s\omega)ds}\int_0^t
e^{-2\xi(\theta_s\omega)+\delta s
-2\int_0^s\xi(\theta_r\omega)dr}ds,
\end{split}
\end{equation*}
where $c_1=\frac{\max\{1, \frac{1}{\varrho}\}}{\delta\min\{1,
\frac{1}{\varrho}\}}$. Denote
$$a(\omega)=2\int_0^T|\xi(\theta_s\omega)|ds$$
and
\begin{equation*}
\begin{split}
b(\omega)&=c_1\max_{t\in [0, T]}\{(\|h\|^2+\|g\|^2)\\
&\cdot e^{-\delta t+2\int_0^t\xi(\theta_s\omega)ds}\int_0^t
e^{-2\xi(\theta_s\omega)+\delta s
-2\int_0^s\xi(\theta_r\omega)dr}ds\}.
\end{split}
\end{equation*}
Due to the properties of the Ornstein-Uhlenbeck process, we know
that $a(\omega), b(\omega)$ are well-defined. Then we have
\begin{equation*}
\|\Psi(t)\|^2\leq  \|\Psi_0\|^2e^{a(\omega)}+b(\omega),
\end{equation*}
which implies that the solution $\Psi$ is defined in any interval
$[0, T]$.

(ii) Let $\Phi_0=(\bar{U}_0, \bar{V}_0)$, $\Psi_0=(U_0, V_0)\in
\mathbb{E}$, and $\Phi(t):=(\bar{U}(t, \omega, \bar{U}_0),
\bar{V}(t, \omega, \bar{V}_0)), \Psi(t):=(U(t, \omega, U_0), V(t,
\omega, V_0))$ be two solutions of \eqref{sys3}. By denoting
$Z(t)=\Phi(t)-\Psi(t)$, we have
\begin{equation*}
\begin{split}
\frac{d}{dt_+}&\|Z(t)\|^2\\
& \leq 2e^{-\xi(\theta_t\omega)} \|f(e^{\xi(\theta_t\omega)}\Phi(t))
-f(e^{\xi(\theta_t\omega)}\Psi(t))\|\|Z\|\\
&\quad\quad+2\xi(\theta_t\omega)\|Z\|^2\\
&\quad\quad\quad\quad\leq2(L_{Y'}+\xi(\theta_t\omega))\|Z\|^2\leq
\kappa\|Z\|^2,
\end{split}
\end{equation*}
where $\kappa=2(L_{Y'}+\max_{t\in [0, T]}|\xi(\theta_t\omega)|)$ is
well-defined, and $L_{Y'}$ denotes the Lipschitz constant of $f$
corresponding to a bounded set $Y'\in \mathbb{E}$ where $\Phi$ and
$\Psi$ belong to. By the Gronwall lemma again, we obtain
$$\|Z(t)\|^2\leq e^{\kappa t}\|Z(0)\|^2,$$
and consequently
$$\sup_{t\in [0, T]}\|\Phi(t)-\Psi(t)\|^2\leq e^{\kappa T}\|\Phi_0-\Psi_0\|^2.$$
If $\Phi_0=\Psi_0$, then the above inequality indicates that the
uniqueness and continuous dependence on the initial data of the
solutions of \eqref{sys3}.

(iii) The continuity of $\varphi$ is due to (i) and (ii). The
measurability of $\psi$ follows from the properties of $\Lambda$.
Here, we only remain to prove the conjugacy between $\varphi$ and
$\psi$. The verification by chain rule is routine and thus be
omitted. The proof is complete.

\end{proof}

\section{Existence of a global random attractor}

In this section, we will prove the existence of a global random
attractor for system \eqref{sys1}. Since the random dynamical
systems $\varphi$ and $\psi$ are conjugated, we only have to
consider the RDS $\varphi$. Firstly, we have our main result

\begin{theorem} \label{existence}
The SLDS $\varphi$ generated by system \eqref{sys3} has a unique
global random attractor.
\end{theorem}

In order to prove Theorem \ref{existence}, we will use Proposition
\ref{condition}. We first need to prove there exists an absorbing
set for $\varphi$ in $\mathcal{D}(\mathbb{E})$. Next, we will show
the RDS $\varphi$ is random asymptotically null in the sense of
\eqref{asy null}.

\begin{lemma} \label{Absorbing}
There exists a closed random tempered set $\mathcal{K}(\omega)\in $
$\mathcal{D}(\mathbb{E})$ such that for all $B\in
\mathcal{D}(\mathbb{E})$ and a.e. $\omega \in \Omega$ there exists
$t_{B}(\omega)>0$ such that
$$\varphi(t,\theta_{-t}\omega)B(\theta_{-t}\omega)\subset
\mathcal{K}(\omega) \ \ \mbox{for all} \ \ t\geq t_{B}(\omega).$$
\end{lemma}

\begin{proof}
Let us start with $\Psi(t)=\varphi(t, \omega, \Psi_0)$. Then by
\eqref{est1}, we have
\begin{equation*}
\begin{split}
\|\varphi(t)&\|^2\leq e^{-\delta t+2\int_0^t\xi(\theta_s\omega)ds}
\|\Psi_0\|^2+c_1(\|h\|^2+\|g\|^2)\\
&\cdot e^{-\delta t+2\int_0^t\xi(\theta_s\omega)ds}\int_0^t
e^{-2\xi(\theta_s\omega)+\delta s
-2\int_0^s\xi(\theta_r\omega)dr}ds.
\end{split}
\end{equation*}
Now, by replacing $\omega$ with $\theta_{-t}\omega$ and $\Psi_0$
with $e^{-\xi(\theta_{-t}\omega)}\Psi_0$, respectively, in the
expression $\varphi$, we obtain
\begin{equation}\label{est2}
\begin{split}
 &\|\varphi(t, \theta_{-t}\omega,
e^{-\xi(\theta_{-t}\omega)}\Psi_0)
\|^2\\
&\leq e^{-\delta t+2\int_0^t\xi(\theta_{s-t}\omega)ds}\|
e^{-\xi(\theta_{-t}\omega)}\Psi_0\|^2\\
&\quad\quad+c_1(\|h\|^2+\|g\|^2) e^{-\delta
t+2\int_0^t\xi(\theta_{s-t}\omega)ds}\\
&\quad\cdot\int_0^ te^{-2\xi(\theta_{s-t}\omega)+\delta s
-2\int_0^s\xi(\theta_{r-t}\omega)dr}ds\\
&\leq e^{-\delta t-2\xi(\theta_{-t}\omega)+2\int_0^t
\xi(\theta_{s-t}\omega)ds}\|\Psi_0\|^2\\
&\quad\quad+c_1(\|h\|^2+\|g\|^2)\\
&\quad\cdot\int_0^te^{-2\xi(\theta_{s-t}\omega)+\delta (s-t)
-2\int_s^t\xi(\theta_{r-t}\omega)dr}ds\\
&\leq e^{-\delta t-2\xi(\theta_{-t}\omega)+2\int_{-t}^0
\xi(\theta_{s}\omega)ds}\|\Psi_0\|^2\\
&\quad\quad+c_1(\|h\|^2+\|g\|^2)\\
&\quad \cdot\int_{-t}^0e^{-2\xi(\theta_{s}\omega)+\delta s
-2\int_s^0\xi(\theta_{r}\omega)dr}ds\\
&\leq e^{-\delta t-2\xi(\theta_{-t}\omega)+2\int_{-t}^0
\xi(\theta_{s}\omega)ds}\|\Psi_0\|^2\\
&\quad\quad+c_1(\|h\|^2+\|g\|^2)\\
&\quad\cdot\int_{-\infty}^0e^{-2\xi(\theta_{s}\omega)+\delta s
-2\int_s^0\xi(\theta_{r}\omega)dr}ds.
\end{split}
\end{equation}
By the properties of the Ornstein-Uhlenbeck process, we know that
$$\int_{-\infty}^0e^{-2\xi(\theta_{s}\omega)+\delta s
-2\int_s^0\xi(\theta_{r}\omega)dr}ds<+\infty.$$ Consider for any
$\Psi_0\in B(\theta_{-t}\omega)$, we have
\begin{eqnarray*}
&&\|\varphi(t, \theta_{-t}\omega,
e^{-\xi(\theta_{-t}\omega)}\Psi_0)\|^2\\
&\leq& e^{-\delta t-2\xi(\theta_{-t}\omega)+2\int_{-t}^0
\xi(\theta_{s}\omega)ds}
d(B(\theta_{-t}\omega))^2\nonumber\\
&&+c_1(\|h\|^2+\|g\|^2)
\int_{-\infty}^0e^{-2\xi(\theta_{s}\omega)+\delta s
-2\int_s^0\xi(\theta_{r}\omega)dr}ds.
\end{eqnarray*}
Note that
\begin{eqnarray*}
\lim_{t\rightarrow+\infty}e^{-\delta t-2
\xi(\theta_{-t}\omega)+2\int_{-t}^0
\xi(\theta_{s}\omega)ds}d(B(\theta_{-t}\omega))^2=0,
\end{eqnarray*}
and denote
\begin{eqnarray*}
\begin{split}
R^2(\omega)=&1+c_1(\|h\|^2+\|g\|^2)\\
&\quad\cdot\int_{-\infty}^0e^{-2\xi(\theta_{s}\omega)+\delta s
-2\int_s^0\xi(\theta_{r}\omega)dr}ds,
\end{split}
\end{eqnarray*}
we conclude that
\begin{eqnarray}\label{abs set}
\mathcal{K}(\omega)=\overline{B_{\mathbb{E}}(0, R(\omega))}
\end{eqnarray}
is an absorbing closed random set. It remains to show that
$\mathcal{K}(\omega)\in \mathcal{D}(\mathbb{E})$. Indeed, from
Definition \ref{tempered random set} (iv), for all $\gamma>0$, we
get
\begin{eqnarray*}
\begin{split}
e^{-\gamma t}&R^2(\theta_{-t}\omega)=e^{-\gamma t}+c_1e^{-\gamma t}
(\|h\|^2+\|g\|^2)\\
&\cdot\int_{-\infty}^0e^{-2\xi(\theta_{s-t}\omega)+\delta s
-2\int_s^0\xi(\theta_{r-t}\omega)dr}ds\\
&\quad\quad\quad=e^{-\gamma t}+c_1e^{-\gamma t}(\|h\|^2+\|g\|^2)\\
&\cdot\int_{-\infty}^{-t}e^{-2\xi(\theta_{s}\omega)+\delta (s+t)
-2\int_s^0\xi(\theta_{r}\omega)dr}ds\rightarrow 0\\
&\quad\quad\quad\mbox{as} \ \ t\rightarrow\infty,
\end{split}
\end{eqnarray*}
which completes the proof.
\end{proof}

\begin{lemma} \label{Asymptotic Null}
Let $\Psi_0(\omega)\in \mathcal{K}(\omega)$ be the absorbing set
given by \eqref{abs set}. Then for every $\epsilon>0$, there exist
$\tilde{T}(\epsilon, \omega, \mathcal{K}(\omega))>0$ and
$\tilde{N}(\epsilon, \omega, \mathcal{K}(\omega))>0$, such that the
solution $\varphi$ of problem \eqref{sys3} is random asymptotically
null, that is, for all $t\geq \tilde{T}(\epsilon, \omega,
\mathcal{K}(\omega))$,
\begin{eqnarray*}
\begin{split}
\sup_{\Psi\in \mathcal{K}(\omega)}\sum_{|i|>\tilde{N} (\epsilon,
\omega, \mathcal{K}(\omega))}|& \varphi_i(t, \theta_{-t}\omega,
\Psi(\theta_{-t}\omega)|^2\leq \epsilon^2.
\end{split}
\end{eqnarray*}
\end{lemma}

\begin{proof}
Choose a smooth cut-off function satisfying $0\leq \rho(s)\leq 1$
for $s\in \mathbb{R^{+}}$ and $\rho(s)=0$ for $0\leq s\leq 1$,
$\rho(s)=1$ for $s\geq 2$. Suppose there exists a positive constant
$c_0$ such that $|\rho'(s)|\leq c_0$ for $s\in \mathbb{R}^+$.

Let $N$ be a fixed integer which will be specified later, set
$x=(\rho(\frac{|i|}{N})U_{i})_{i \in \mathbb{Z}}$ and
$y=(\rho(\frac{|i|}{N})V_{i})_{i \in \mathbb{Z}}$. Then take the
inner product of the two equations in system \eqref{sys3} with $x$
and $y$ in $\ell^2$, respectively, and combine the following two
inequalities
\begin{eqnarray*}
(AU, x)=(\tilde{B}U, \tilde{B}x)\geq -\frac{2c_0}{N}\|U\|^2 \geq
-\frac{2c_0}{N}\|\varphi\|^2,
\end{eqnarray*}
and
\begin{eqnarray*}
-\infty<-2e^{-\xi(\theta_{t}\omega)} \sum_{i\in
\mathbb{Z}}\rho(\frac{|i|}{N})f_i(e^{\xi (\theta_{t}\omega)}U_i)
U_i\leq 0,
\end{eqnarray*}
we have
\begin{eqnarray*}
&&\frac{d}{dt_+}\sum_{i\in \mathbb{Z}}\rho(\frac{|i|}{N})|
\varphi_{i}|^2+(\delta-2\xi(\theta_{t}\omega))
\sum_{i\in \mathbb{Z}}\rho(\frac{|i|}{N})|\varphi_{i}|^2\\
&&\quad\quad\leq \frac{c_2}{N}\|\varphi(t, \omega,
e^{-\xi(\omega)}\Psi_0)\|^2\\
&&\quad\quad\quad\quad+c_1e^{-\xi(\theta_{t}\omega)} \sum_{|i|\geq
N}(|h_i|^2+|g_i|^2),
\end{eqnarray*}
where $c_2=\frac{4c_0}{\min\{1, \frac{1}{\varrho}\}}$. By using the
Gronwall lemma, for $t\geq T_{\mathcal{K}}
=T_{\mathcal{K}}(\omega)$, it follows that
\begin{eqnarray}
&&\sum_{i\in \mathbb{Z}}\rho(\frac{|i|}{N})|\varphi_{i}
(t, \omega, e^{-\xi(\omega)}\Psi_0(\omega))|^2\nonumber\\
&\leq & e^{-\delta(t-T_{\mathcal{K}})+2\int_{T_{\mathcal{K}}}^t
\xi(\theta_{s}\omega)ds}\nonumber\\
&&\quad\quad\cdot\sum_{i\in \mathbb{Z}} \rho(\frac{|i|}{N})
|\varphi_{i}(t, \omega, e^{-\xi(\omega)}
\Psi_0(\omega))|^2\label{est3}\\
&&+\frac{c_2}{N}\int_{T_{\mathcal{K}}}^t
e^{-\delta(t-\tau)+2\int_{\tau}^t\xi(\theta_{s}\omega)ds}\nonumber\\
&&\quad\quad\quad\quad\cdot\|\varphi(\tau, \omega, e^{-\xi(\omega)}
\Psi_0(\omega))\|^2d\tau\label{est4}\\
&&+c_1\sum_{|i|\geq N}(|h_i|^2+|g_i|^2)\nonumber\\
&&\quad\quad\cdot\int_{T_{\mathcal{K}}}^t
e^{-\delta(t-\tau)+2\int_{\tau}^t
\xi(\theta_{s}\omega)ds-\xi(\theta_{t}\omega)}d\tau.\label{est5}
\end{eqnarray}
Now, substitute $\theta_{-t}\omega$ for $\omega$ and estimate each
term from \eqref{est3} to \eqref{est5}. In \eqref{est2}, with $t$
replaced with $T_{\mathcal{K}}$ and $\omega$ with
$\theta_{-t}\omega$, respectively, it follows from \eqref{est3} that
\begin{eqnarray*}
&&e^{-\delta(t-T_{\mathcal{K}})+2\int_{T_{\mathcal{K}}}^t
\xi(\theta_{s-t}\omega)ds}\nonumber\\
&&\quad\cdot\sum_{i\in \mathbb{Z}}\rho(\frac{|i|}{N})
|\varphi_{i}(T_{\mathcal{K}}, \theta_{-t}\omega,
e^{-\xi(\theta_{-t}\omega)}
\Psi_0(\theta_{-t}\omega))|^2\nonumber\\
&\leq & e^{-\delta t-2\xi(\theta_{-t}\omega) +2\int_{0}^t
\xi(\theta_{s}\omega)ds}\|\Psi_0\|^2\nonumber\\
&&~~+c_1
\int_{0}^{T_{\mathcal{K}}}e^{-2\xi(\theta_{s-t}\omega)+\delta (s-t)
+2\int_s^t\xi(\theta_{r-t}\omega)dr}ds\nonumber\\
&\leq &e^{-\delta t-2\xi(\theta_{-t}\omega) +2\int_{0}^t
\xi(\theta_{s}\omega)ds}\|\Psi_0\|^2\nonumber\\
&&~~+c_1 \int_{-t}^{T_{\mathcal{K}}-t}e^{-2
\xi(\theta_{s}\omega)+\delta s -2\int_s^0\xi(\theta_{r}\omega)dr}ds.
\end{eqnarray*}
Due to the properties of the Ornstein-Uhlenbeck process, there
exists a $T_1(\epsilon, \omega,
\mathcal{K}(\omega))>T_{\mathcal{K}}(\omega)$, such that if
$t>T_1(\epsilon, \omega, \mathcal{K}(\omega))$, then
\begin{eqnarray}
&&e^{-\delta(t-T_{\mathcal{K}})+2\int_{T_{\mathcal{K}}}^t
\xi(\theta_{s-t}\omega)ds}\nonumber\\
&&\quad\cdot\sum_{i\in \mathbb{Z}}\rho(\frac{|i|}{N})
|\varphi_{i}(T_{\mathcal{K}}, \theta_{-t}\omega,
e^{-\xi(\theta_{-t}\omega)}
\Psi_0(\theta_{-t}\omega))|^2\nonumber\\
&&\quad\quad\quad\quad \leq \frac{\epsilon^2}{3}. \label{res1}
\end{eqnarray}
Next, from \eqref{est2} and \eqref{est4}, it follows that
\begin{eqnarray*}
&&\frac{c_2}{N}\int_{T_{\mathcal{K}}}^te^{-\delta(t-\tau)
+2\int_{\tau}^t\xi(\theta_{s-t}\omega)ds}\nonumber\\
&&\quad\quad\cdot \|\varphi(\tau, \theta_{-t}\omega, e^{
-\xi(\theta_{-t}\omega)}\Psi_0(\theta_{-t}\omega))\|^2d\tau\nonumber\\
&\leq & \frac{c_2}{N}\|\Psi_0\|^2(t-T_{\mathcal{K}}) e^{-\delta
t-2\xi(\theta_{-t}\omega)+2\int_{0}^t
\xi(\theta_{s-t}\omega)ds}\nonumber\\
&&+\frac{c_1c_2}{N}(\|h\|^2+\|g\|^2)\nonumber\\
&&\quad\cdot\int_{T_{\mathcal{K}}}^t\int_0^{\tau}e^{-2
\xi(\theta_{s-t}\omega)+\delta (s-t)
+2\int_s^t\xi(\theta_{r-t}\omega)dr}dsd\tau\nonumber\\
&\leq & \frac{c_2}{N}\|\Psi_0\|^2(t-T_{\mathcal{K}}) e^{-\delta
t-2\xi(\theta_{-t}\omega)+2\int_{0}^t
\xi(\theta_{s-t}\omega)ds}\nonumber\\
&&+\frac{c_1c_2}{N}(\|h\|^2+\|g\|^2)\nonumber\\
&&\quad\cdot\int_{T_{\mathcal{K}}}^t\int_{-t}^{\tau-t}
e^{-2\xi(\theta_{s}\omega)+\delta s
-2\int_s^0\xi(\theta_{r}\omega)dr}dsd\tau.
\end{eqnarray*}
Thanks to the properties of the Ornstein-Uhlenbeck process, there
exist $T_2(\epsilon, \omega,
\mathcal{K}(\omega))>T_{\mathcal{K}}(\omega)$ and $N_1(\epsilon,
\omega, \mathcal{K}(\omega))>0$ such that if $t>T_2(\epsilon,
\omega, \mathcal{K}(\omega))$ and $N>N_1(\epsilon, \omega,
\mathcal{K}(\omega))$, then
\begin{eqnarray}
&&\frac{c_2}{N}\int_{T_{\mathcal{K}}}^t
e^{-\delta(t-\tau)+2\int_{\tau}^t\xi(\theta_{s-t}\omega)ds}\nonumber\\
&&\cdot\|\varphi(\tau, \theta_{-t}\omega, e^{-\xi
(\theta_{-t}\omega)}\Psi_0(\theta_{-t}\omega))\|^2d\tau \leq
\frac{\epsilon^2}{3}. \label{res2}
\end{eqnarray}
Since $h, g\in \ell^2$, by the properties of the Ornstein-Uhlenbeck
process again, we find that there exists $N_2(\epsilon, \omega,
\mathcal{K}(\omega))>0$ such that if $N>N_2(\epsilon, \omega,
\mathcal{K}(\omega))$, then from \eqref{est5},
\begin{eqnarray}
&&c_1\sum_{|i|\geq
N}(|h_i|^2+|g_i|^2)\nonumber\\
&&\quad\cdot\int_{T_{\mathcal{K}}}^t
e^{-\delta(t-\tau)+2\int_{\tau}^t\xi(\theta_{s}\omega)ds
-\xi(\theta_{t}\omega)}d\tau \leq  \frac{\epsilon^2}{3}.
\label{res3}
\end{eqnarray}
Let
\begin{eqnarray*}
&&\tilde{T}(\epsilon, \omega, \mathcal{K}(\omega))=
\max\{T_1(\epsilon, \omega, \mathcal{K}(\omega)),
T_2(\epsilon, \omega, \mathcal{K}(\omega))\},\\
&&\tilde{N}(\epsilon, \omega, \mathcal{K}(\omega))=
\max\{N_1(\epsilon, \omega, \mathcal{K}(\omega)), N_2(\epsilon,
\omega, \mathcal{K}(\omega))\}.
\end{eqnarray*}
Then from \eqref{res1}, \eqref{res2} and \eqref{res3}, for
$t>\tilde{T}(\epsilon, \omega, \mathcal{K}(\omega))$ and
$N>\tilde{N}(\epsilon, \omega, \mathcal{K}(\omega))$, we get
\begin{eqnarray*}
&&\sum_{|i|\geq 2N}|\varphi_{i}(t, \theta_{-t}\omega,
e^{-\xi(\theta_{-t}\omega)}\Psi_0(\theta_{-t}\omega))|^2\\
&\leq & \sum_{i\in \mathbb{Z}}\rho(\frac{|i|}{N})| \varphi_{i}(t,
\theta_{-t}\omega,
e^{-\xi(\theta_{-t}\omega)}\Psi_0(\theta_{-t}\omega))|^2 \leq
\epsilon^2,
\end{eqnarray*}
which implies the conclusion.
\end{proof}

We are now in a position to prove our main result.

\textit{Proof of Theorem \ref{existence}.} The desired result
follows directly from Lemmas \ref{Absorbing} and \ref{Asymptotic
Null} and Proposition \ref{condition}. $\blacksquare$

\begin{remark} The result may have generalized the existing results
(see e.g. \cite{Huang, Gu4}) to some extent. First, c\`adl\`ag
functions in a more wider sense than continues ones as indicated in
Introduction section; Second, here we restrict to $1<\alpha<2$, when
$\alpha=2$, the $\alpha$-stable process actually reduces to the
standard Brownian motion.
\end{remark}

\begin{remark} Recently,
some sufficient conditions for the upper-semicontinuity of
attractors for random lattice systems perturbed by small white
noises have been given in \cite{Zhou}. Here, it is worth mentioning
that all the results on this topic are focus on the SLDS perturbed
by the white noises. It will be an interesting question left to
future research.
\end{remark}

\end{document}